\documentclass[a4paper,12pt]{amsart}

\usepackage[latin1]{inputenc}
\usepackage[T1]{fontenc}
\usepackage[english]{babel}

\usepackage{mathrsfs}
\usepackage{amscd}
\usepackage{amsfonts}
\usepackage{amsmath}
\usepackage{amssymb}
\usepackage{amstext}
\usepackage{amsthm}
\usepackage{amsbsy}

\usepackage{xspace}
\usepackage[all]{xy}
\usepackage{graphicx}
\usepackage{url}
\usepackage{latexsym}

\makeatletter
\newcommand*{\rom}[1]{\expandafter\@slowromancap\romannumeral #1@}
\makeatother

\theoremstyle{definition}

\newtheorem{fact}{fact}

\newtheorem{thm}[fact]{Theorem}
\newtheorem{lemma}[fact]{Lemma}
\newtheorem{prop}[fact]{Proposition}
\newtheorem{corollary}[fact]{Corollary}
\newtheorem{defini}[fact]{Definition}

\title{Models of $Th(\mathbb{N})$ are $IPs$ of nice $RCFs$}
\author{Merlin Carl}

\begin{document}

\begin{abstract}
 Exploring further the connection between exponentiation on real closed fields and the existence of an integer part modelling strong fragments of arithmetic,
we demonstrate that each model of true arithmetic is an integer part of an exponential real closed field that is elementary equivalent to the reals with exponentiation.
\end{abstract}

\maketitle

\section{Introduction}

This work originates in \cite{DKS}, where it was shown that a countable real closed field $K$ has an integer part modelling $PA$ iff it is recursively saturated.
 Marker (see \cite{MS}) gave a counterexample in the uncountable case by lifting exponentiation from the model to every real closed field of which it is an integer part.
This was refined in \cite{CDK} to the theorem that real closed fields with $IP$s modelling $I\Delta_{0}+EXP$ always allow left exponentiation. It is natural to ask what influence a model of arithmetic
has on the spectrum of real closed fields of which it is an $IP$. We show that models of true arithmetic are always $IP$s of real closed fields that are very similar to the reals in a model-theoretic sense.
We also show that this fails if one replaces true arithmetic with bounded arithmetic ($I\Delta_{0}$). We conjecture that Peano arithmetic is actually enough to achieve our results. We don't know
where the exact benchmark is.\\

Notation: If $\vec{v}=(v_1,...,v_n)$, $\vec{v}\in M$ means that $\vec{v}$ is a sequence of elements of $M$.

\section{The $M$-definable reals}

The idea behind the following construction is to define $M$-reals as equivalence classes of convergent sequences of elements of the fraction field $\text{ff}(M)$ of $-M\cup M$.

\begin{defini} Let $M\models Th(\mathbb{N})$. A pre-real over $M$ is a function $f:M\rightarrow (M\cup -M)\times M-\{0\}$ with definable graph, i.e. such that there is an $\mathfrak{L}_{PA}$-formula 
$\psi(x,y,z,\vec{p})$ and a finite sequence $\vec{v}\subseteq M$  such that $M\models\psi(x,y,z,\vec{v})$ iff $f(x)=(y,z)$.\end{defini}

\textbf{Remark}: Strictly speaking, this would not allow the first element of an element of the image to be in $-M$. This can be solved by a convention stating e.g. that $2n+1$ denotes $-n$ while $2n$ denotes $n$.
Since this does not cause any principal difficulties, we will, by slight abuse of notation, ignore this subtlety.\\
Also, when $M$ is clear from the context, we will drop "over $M$". The mentioning of the parameter sequence $\vec{v}$ will usually also be dropped.

\begin{defini} A pre-real over $M$ given by some formula $\psi(x,y,z)$ is zero iff $M\models\forall{m}\exists{n}\forall{k>n}(\psi(k,a,b)\implies ma<b)$. 
It is convergent iff \begin{center} $M\models\forall{m}\exists{n}\forall{k_{1},k_{2}>n}(\psi(k_{1},a_{1},b_{1})\wedge\psi(k_{2},a_{2},b_{2})\implies m(a_{1}b_{2}-a_{2}b_{1})<a_{2}b_{2})$.\end{center} 
A convergent pre-real over $M$ is an $M$-real. Two $M$-reals $x_{1},x_{2}$ given by $\psi_{1}$ and $\psi_{2}$ are equivalent, written $x_{1}\sim x_{2}$ iff
$M\models\forall{m}\exists{n}\forall{k>n}(\psi_{1}(k,a,b)\wedge\psi_{2}(k,c,d)\implies m(ad-bc)<bd)$. Let $[x]_{\sim}$ denote the $\sim$-equivalence class of $x$ when $x$ is an $M$-real.
Finally, we set $K_{M}:=\{[x]_{\sim}|x$ is an $M$-real$\}$. If $n\in M$, then $n_{K}$ denotes the equivalence class of the constant function on $M$ which takes the value $n$ everywhere; the subscript
$K$ is dropped wherever possible.\end{defini}

\begin{defini} Let $x$ and $y$ be $M$-reals. Then we write $x<y$ iff there exist $m,k\in M$ such that, for all $l>k\in M$, $mx_{l}+1<my_{l}$.\end{defini}

From now on, we are almost exclusively interested in arithmetic formulas $\phi(\vec{v},x,y)$ that define an $M$-real for every $\vec{v}$, i.e. such that 
\begin{center} $Th(\mathbb{N})\models\forall{\vec{v}}`\phi(\vec{v},x,y)$ \end{center}
defines
a total function from naturals to pairs of naturals with second element $\neq0$ and this function gives rise to a convergent sequence'. Let us call such a formula $\phi$ a `safe' formula.\\
What we would like to do is assume that all occuring formulas, unless stated otherwise, are of this kind. However, we have to ensure that by this restriction, we do not loose any $M$-reals.

\begin{lemma}
 For any $\mathfrak{L}_{PA}$-formula $\phi(\vec{v},x,y)$, there exists an $\mathfrak{L}_{PA}$-formula $\phi^{\prime}(\vec{v},x,y)$ such that it is a theorem of $Th(\mathbb{N})$ that, for every parameter $\vec{v}$, $\phi$ and $\phi^{\prime}$ 
define the same function if $\phi$ defines a convergent total function and otherwise $\phi^{\prime}$ defines the constant $0$ function.
\end{lemma}
\begin{proof}
 We abbreviate by $tot(\phi,\vec{v})$ the $\mathfrak{L}_{PA}$-formula expressing that $\phi(\vec{v},x,y)$ defines a total function such that $\phi(\vec{v},a,b)\wedge\pi_{2}(b)=c$ implies that $b\neq 0$. (Here $\pi_{2}$ is the function for obtaining
the second element of a coded pair.) Then, $conv(\phi,\vec{v})$ expresses that $tot(\phi,\vec{v})$ and that $((a,b)|x\in M\wedge\phi(\vec{v},a,b))$ defines a convergent sequence.
Now let $\phi^{\prime}(\vec{v},x,y)$ be \begin{center}$(conv(\phi,\vec{v})\wedge\phi(\vec{v},x,y))\vee(\neg conv(\phi,\vec{v})\wedge \pi_{1}(y)=0\wedge\pi_{2}(y)=1)$\end{center}. This is obviously as desired.
\end{proof}

\begin{corollary}
 If $x$ is an $M$-real, then there exist a safe formula $\phi$ and a finite sequence $\vec{v}\subseteq M$ such that $\phi(\vec{v},i,j)$ defines $x$.
\end{corollary}
\begin{proof}
 Immediate from the last lemma. (Take the corresponding safe formula.)
\end{proof}

\begin{prop}
There is an $\mathfrak{L}_{PA}[X,Y]$-formula $\phi_{<}(X,Y)$ such that, for all $x,y\in K_{M}$, $\phi(X\mapsto x,Y\mapsto y)$ holds iff $x<y$.
\end{prop}
\begin{proof}
 Immediate from the definition.
\end{proof}

\begin{defini}
 Let $x=(x_{i})_{i\in M}$ and $y=(y_{i})_{i\in M}$ be $M$-reals with $x_{i}=\frac{a_{i}}{b_{i}}$ and $y_{i}=\frac{c_{i}}{d_{i}}$. We define $x+_{M}y$ by $(x_{i}+y_{i})_{i\in M}$, where
$x_{i}+y_{i}=\frac{a_{i}d_{i}+b_{i}c_{i}}{b_{i}d_{i}}$. Furthermore, we define $x\cdot_{M}y$ by $(x_{i}y_{i})_{i\in M}$, where $x_{i}\cdot y_{i}=\frac{a_{i}c_{i}}{b_{i}d_{i}}$. The subscript
$M$ is dropped whenever there is no danger of confusion.
\end{defini}

\begin{prop}
 $K_{M}$ is closed under $+$ and $\cdot$.
\end{prop}
\begin{proof}
 Trivial.
\end{proof}

\begin{lemma} $(K_{M},+,\cdot,<)$ is an ordered field.\end{lemma}
\begin{proof}
It is clear from the definition that $(K_{M},+,<)$ and\\ $(K^{M}-[0]_{\sim},\cdot,<)$ are ordered abelian groups. The distributivity of $\cdot$ over $+$ is also immediate.\\
We proceed by showing that, for all $x\in K_{M}$, we have $x>0$ iff there exists $y$ such that $x=y^2$.\\
To see this, let $x\in K_{M}^{>0}$ be arbitrary, say $x=(\frac{p_{i}}{q_{i}})_{i\in M}$. As $x>0$ and $x$ is convergent, there must exist some $m\in M$ such that $p_{i}>0$ for $i>m$.
As $M\models Th(\mathbb{N})$, it holds in $M$ that, for every $k\in M^{>0}$, there exists $k^{\prime}$ such that $k^{\prime2}\leq k<(k^{\prime}+1)^2$.
Let $x^{\prime}:=(\frac{p_{i+m}^{\prime}}{q_{i+m}^{\prime}})_{i\in M}$. Then $M\models |\frac{p_{i+m}}{q_{i+m}}-(\frac{p_{i}^{\prime}}{q_{i}^{\prime}})^2|<\frac{3}{q_{i}}$. Since $x$ is convergent,
$(\frac{3}{q_{i}})_{i\in M}$ is also convergent, hence $x^{\prime2}\sim x$. So $x^{\prime}$ is as desired.\\
In order to see that $K_{M}$ is an ordered field, we finally show that $-1$ is not a sum of squares. Otherwise, let $-1=x_{1}^2+...+x_{n}^2$ with $x_{1},...,x_{n}\in K_{M}$ By definition of $K_{M}$, there are formulas $\phi_{1},...,\phi_{n}$ and parameters
$\vec{v}_{1},...,\vec{v}_{n}$ such that $\phi_{i}(\vec{v}_{i},x)$ codes the $M$-real $x_{i}$. 
Hence $M\models\exists\vec{v}_{1},...,\vec{v}_{n}(x_{1}^{2}+...+x_{n}^{2}=-1)$ (the term in the brackets appropriately
expressed). By elementary equivalence, $\mathbb{N}$ is a model of the same statement. Hence $-1$ is a sum of squares in the reals, a contradiction.\end{proof}

\begin{thm}\label{recclosed}
Let $X_1,...,X_n \in K_{\mathbb{N}}$, and let $Y:\mathbb{N}\rightarrow\mathbb{Z}\times\mathbb{N}^{>0}$.\\
(1) If $Y$ is recursive in $X_1,...,X_n$ and convergent, then $Y\in K_{\mathbb{N}}$.\\
(2) $K_{\mathbb{N}}$ is closed under the Turing-Jump, i.e. for $Y\in K_{\mathbb{N}}$, $n\in\mathbb{N}$, we have $Y^{(n)}\in K_{\mathbb{N}}$.
\end{thm}
\begin{proof}
(1) Let $P$ be a Turing programm such that $P^{\oplus_{i=1}^{n}X_{i}}(k)=y_{k}$ (the $k$-th bit of $y$) for all $k\in\mathbb{N}$. Let $\phi_{P}(v_1,v_2,X_1,...,X_n)$ be a formula of $\mathfrak{L}_{PA}$ amended
with $n$ extra predicates such that, for all $i,j\in\mathbb{N}$, $Z_1,...,Z_n\in\mathbb{R}$, $\phi_{P}(i,j,Z_1,...,Z_n)$ holds in $\mathbb{N}$ iff $P^{\oplus_{i=1}^{n}Z_{i}}(i)\downarrow j$. Now consider $\tilde{\phi}_{P}(x,y)$ obtained
by eliminating the $X_i$ using their definition in $K_{M}$. (I.e. $X_{1}(t)$ would be replaced by $\exists{\tilde{t}}\phi_{1}(\tilde{t})$, where $\phi_{1}$ defines $X_1$.) Then $\tilde{\phi}_{P}$ is an $\mathfrak{L}_{PA}$-formula defining
$Y$. Hence $Y\in K_{\mathbb{N}}$.\\
(2) By arithmetical definability of the Turing jump. 
\end{proof}

\begin{prop} $M_{K}:=\{n_{K}|n\in M\}\subseteq K_{M}$.\end{prop}
\begin{proof} Immediate, as constant functions are obviously definable over $M$.\end{proof}

\begin{prop} $(M_{K},0_{K},1_{K},+_{K},\cdot_{K},<_{K})\equiv_{el}(M,0,1,+,\cdot,<)$.\end{prop}
\begin{proof}: Obvious.\end{proof}

\begin{lemma} $M_{K}$ is an integer part of $K_{M}$.\end{lemma}
\begin{proof}: For $(a,b)\in M\times M-\{0\}$, define $\lfloor\frac{a}{b}\rfloor$ to be the unique $k\in M$ such that $kb\leq a<(k+1)b$. If $\psi$ defines a real $r$ over $M$, then 
$\phi(x)\equiv \forall{n}\exists{k>n}\exists{a,b}(\psi(k,a,b)\wedge x=\lfloor\frac{a}{b}\rfloor)$ defines a subset $S$ of $M$ (which is clearly non-empty, as $\lfloor\frac{a}{b}\rfloor$ exists
for all $a,b\in M$ since $M\models Th(\mathbb{N})$). As $M$ is a model of true arithmetic and hence of full induction, 
$S$ must have a least element $s$. By definition, there must be $k^{\prime}\in M$ such that from $k^{\prime}$ on, the floor functions of the elements of $r$ never drop below $s$. 
Also, there is some $k^{\prime\prime}$ such that, from $k^{\prime\prime}$ on, the elements of $r$ are at most $\frac{1}{2}$ apart. If $k>max\{k^{\prime},k^{\prime\prime}\}$, it follows
that from $k$ on, the only possible values of the floor function are $s$ and $s+1$. We now distinguish the following cases:\\
(1) From some point on, the floor function becomes constantly $s$. Then all elements or $r$ eventually lie between $s$ and $s+1$, hence $s_{K}\leq r<s_{K}+_{K}1_{K}$.\\
(2) The floor function alternates cofinally many times between $s$ and $s+1$. As $r$ converges, this implies that the elements of $r$ get arbitrarily close to $s+1$, so that
$r\sim (s+1)_{K}$.\\
In both cases, $r$ can be rounded down to an element of $M_{K}$.\end{proof}

\begin{prop}: Let $K$ be a real closed field, let $Q$ be a dense subset of $K$, $\varepsilon$ a positive element of $K$ and let $p$ be a polynomial such that, for all $q\in Q$, we have
$p(q)\geq\varepsilon$. Then $p$ has no zero in $K$.\end{prop}
\begin{proof}: As $K$ is an $RCF$, it inherits from $\mathbb{R}$ the property that polynomials are continuous. Hence, when we get arbitrarily close to a zero, the image has to become arbitrarily small, yet,
by assumption, it remains above $\varepsilon>0$, a contradiction.\end{proof}

\textbf{Convention}: If $\phi(x,y,z,\vec{p})$ is an $\mathfrak{L}_{PA}$-formula and $\vec{v}\subseteq M$ is such that $\phi(x,y,z,\vec{v})$ defines an $M$-real, then 
this $M$-real is denoted by $x_{\phi}^{\vec{v}}$.

\begin{lemma}: $K_{M}$ is closed under square roots for positive elements, i.e. if $0<c\in K_{M}$, then there exists $d\in K_{M}$ such that $c=d^2$.\end{lemma}
\begin{proof}: For every $\phi$, there exists $\psi$ such that $\mathbb{N}\models\forall{\vec{v}}\exists{\vec{p}}(x_{\psi}^{\vec{p}})^2=x_{\phi}^{\vec{v}}$ by Theorem \ref{recclosed} 
since the square root of any $x\in\mathbb{R}$ is recursive in $x$. Hence $M$ is a model of the same statement. Now, every $x\in K_{M}$ is defined by some $\phi$ and some parameters from $M$,
it follows that $K_{M}$ is closed under square roots of positive elements.\end{proof}

\begin{lemma} $K_{M}$ is real closed.\end{lemma}
\begin{proof}: It suffices to show that $K_{M}$ is formally real, closed under square roots for positive elements and that, for every $n\in\mathbb{N}$ and $c_{0},...,c_{2n+1}\in K_{M}$ with $c_{2n+1}\neq 0$, 
the polynomial $p(x)=\Sigma_{i=0}^{2n+1}c_{i}x^{i}$ has a root in $K_{M}$. We have already shown that $K_{M}$ is closed under square roots for positive elements and formally real.\\
The proof that polynomials of odd degree have roots is similar to the proof of root-closure for positive elements: Such a root is (over $\mathbb{R}$) recursive in the coefficients of the polynomial. Hence, for every
$n\in\mathbb{N}$ and every sequence $(\phi_{0},...,\phi_{2n+1})$ of formulas, there exists a formula $\psi$ such that we have \begin{center}
$\mathbb{N}\models\forall{\vec{v}_{0},...,\vec{v}_{2n+1}}((x_{\phi_{2n+1}}^{\vec{v}_{2n+1}}\neq 0)\implies(\exists{\vec{v}}(\Sigma_{i=1}^{2n+1}x_{\phi_{i}}^{\vec{v}_{i}}(x_{\phi}^{\vec{v}})^{i}=0)))$.\end{center}
So $M$ is a model of the same statement. Thus every polyomial of odd degree over $K_{M}$ has a root in $K_{M}$. \end{proof}

\section{Functions on $K_{M}$}

In this section, we start considering analysis on $K_{M}$. To this purpose, we need to define functions on $K_{M}$. If properties of these functions are to be preserved between different $K_{M}$s, these will have to 
be sufficiently explicitely definable in $M$. This is made precise by the following definition.

\begin{defini}
For $n\in\mathbb{N}$, $f:K_{M}^{n}\rightarrow K_{M}$ is $M$-definable iff there are $\phi[X_1,...,X_n]\in\mathfrak{L}_{PA}[X_1,...,X_n]$ (language of arithmetic with $n$ extra predicate symbols $X_1,...,X_n$) and $\vec{v}\in M$ such that, for any
 $\vec{x}\in K_{M}^{n}$, $\phi(X_1\mapsto x_1,X_2\mapsto x_2,...,X_n\mapsto x_n,\vec{v},i,j,k)$ defines an $M$-real $y$ such that $f(x)=y$. Denote by $\text{Def}^{n}(M)$ the set of $n$-ary $M$-definable functions and 
let $\text{Def}(M):=\bigcup_{i\in\mathbb{N}}\text{Def}^{i}(M)$.
\end{defini}

\begin{prop}
If $f$ is $M$-definable, then it is, for each $\psi$, uniformly definable (in the parameter $\vec{v}$) for all $M$-reals definable by $\psi$, i.e. there is a formula $\psi^{\prime}$, depending on $\psi$ but not on $\vec{v}$, such that
$\psi^{\prime}(\vec{v},...)$ defines the image of each $x$ if $x$ is of the form $x_{\psi}^{\vec{v}}$ for some $\vec{v}\in M$.
\end{prop}
\begin{proof}
 Simply plug in the definition instead of the second-order variable $X$ from the definition above.
\end{proof}

\begin{prop}
 $\text{Def}(M)$ contains all constant functions and is closed under composition.
\end{prop}
\begin{proof}
 Trivial.
\end{proof}

\begin{defini}
The exponential function $exp_{M}:K_{M}\rightarrow K_{M}$ (with base $2$) is defined as follows:\\
For elements of $M^{>0}$, exponentiation with arbitrary bases is given by the usual arithmetical definition.\\
Now, for $a,b,n\in M^{>0}$, we let $appr(n,a,b)$ be the largest $m\in M$ such that $m^{b}\leq n^{b}a$.\\
Next, for $K_{M}^{>0}\ni x=(\frac{a_i}{b_i})_{i\in M}$, we assume without loss of generality that $a_{i}$ and $b_{i}$ are positive for all $i\in M$ and set\\ $exp(x):=(\frac{appr(i,exp(a_{i}),b_{i})}{i})_{i\in M}$. Finally,
if $x\in K_{M}^{<0}$, we suppose without loss of generality that for all $i$, $a_i<0$ and $b_i>0$ and let\\ $exp_{M}(x)=(\frac{i}{appr(i,exp(a_{i}),b_{i})})_{i\in M}$.
\end{defini}

\textbf{Convention}: Whenever possible without causing confusion, we will drop the subscripts.

\begin{lemma}
$exp_{\mathbb{N}}=exp_{2}|K_{\mathbb{N}}$, where $exp_{2}$ is the usual real exponential function with base $2$.
\end{lemma}
\begin{proof}
Trivial.
\end{proof}

\begin{lemma}{\label{expcont}}
For every $M\models Th(\mathbb{N})$, $exp_{M}$ is continuous.
\end{lemma}
\begin{proof}
We first note that continuity holds on the quotient field of $M$: As $\mathbb{N}\models \forall{\varepsilon>0}\exists{\delta>0}\forall{p,q,r,s}(|\frac{p}{q}-\frac{r}{s}|<\delta\implies |exp(\frac{p}{q})-exp(\frac{r}{s})|<\varepsilon)$, 
$M$ is a model of the same statement.\\
Now we show that, for any $\mathfrak{L}_{PA}$-formula $\phi$, every $\vec{v}\in M$ and every $q\in \text{ff}(M)$, $q<x_{\phi}^{\vec{v}}$ implies $exp_{M}(q)<exp_{M}(x_{\phi}^{\vec{v}})$ and $x_{\phi}^{\vec{v}}<q$ implies 
$exp_{M}(x_{\phi}^{\vec{v}})<exp_{M}(q)$. This follows from the fact that, for all $\phi\in\mathfrak{L}_{PA}$, we have 
$\mathbb{N}\models\forall{\vec{v}}\forall{p}\forall{q\neq0}((x_{\phi}^{\vec{v}}<\frac{p}{q})\implies(exp(x_{\phi}^{\vec{v}_{1}})<exp(\frac{p}{q})))$ so that the same statement holds in $M$ (and similarly for the other inequality).\\
As every $x\in K_{M}$ is presentable as some $x_{\phi}^{\vec{v}}$, it follows that for all $x\in K_{M}$, $q\in \text{ff}(M)$, we have that $x<q$ implies $exp_{M}(x)<exp_{M}(q)$ and that $q<x$ implies $exp_{M}(q)<exp_{M}(x)$.
But now, as $\text{ff}(M)$ is dense in $K_{M}$ (since $M$ is an $IP$ of $K_{M}$), if $x,y\in K_{M}$ are such that $x<y$, then there exists $q\in \text{ff}(M)$ such that $x<q<y$. It follows that $exp_{M}(x)<exp_{M}(q)<exp_{M}(y)$,
so $exp_{M}(x)<exp_{M}(y)$. Hence $exp_{M}$ is monotonic.\\
The proof that $exp_{M}$ is continuous is now quite straightforward: Let $x\in K_{M}$, then $exp(q)<exp(x)<exp(p)$ for all $q,p\in \text{ff}(M)$ with $q<x<p$. Let $\varepsilon>0$ be given. 
Pick $\delta>0$ such that, for all $\frac{p}{q}$ ($p,q\in M$) with $|x-\frac{p}{q}|<\delta$,
we have that $|exp(x)-exp(\frac{p}{q})|<\varepsilon$.\\
To see that such a $\delta$ exists, let $x=x_{\phi}^{\vec{v}}$, where $\phi$ is safe. Clearly, we have
\begin{center}
$\mathbb{N}\models\forall{\vec{p}}\forall{m>0}\exists{n>0}\forall{p,q\neq0}(|x-\frac{p}{q}|<\frac{1}{n}\rightarrow|exp(x)-exp(\frac{p}{q})|<\frac{1}{m})$.
\end{center}
Hence $M$ is a model of the same statement. If we take $0<m\in M$ large enough such that $\frac{1}{m}<\varepsilon$ - which is possible since $M$ is an integer part of $K_{M}$ - and take $\vec{p}=\vec{v}$,
this guarantees the existence of some $\delta\in K_{M}$ as desired.\\
Now, by monotonicity, it holds for $y\in K_{M}\cap]x-\delta,x+\delta[$ that $|exp(x)-exp(y)|<|exp(x)-exp(\frac{a}{b})|<\varepsilon$, where $a,b\in M$ are such that
either $x-\delta<\frac{a}{b}<y<x$ oder $x<y<\frac{a}{b}<x+\delta$. (That such a choice of $\frac{a}{b}$ is always possible is again clear as $\text{ff}(M)$ is dense in $K_{M}$.) Hence $\delta$ is such that, for all
$y\in K_{M}$, $|x-y|<\delta$ implies $exp_{M}(x)-exp_{M}(y)<\varepsilon$. As $x$ and $\varepsilon$ were arbitrary, it follows that $exp_{M}$ is continuous.

\end{proof}

\textbf{Remark}: The monotonicity is crucial in this argument; it can, however, be relaxed for other functions by splitting $K_{M}$ into intervalls on which they are monotonic. This is particularly useful when one wants to turn to other functions.

\begin{lemma}{\label{cont}}
 Let $f_{1},...,f_{n},g$ be $M$-definable continuous functions. Then $g(f_1,...,f_n)$ is also $M$-definable and continuous. Consequently, every function obtained from $+,\cdot,exp$ by composition is continuous.
\end{lemma}
\begin{proof}
 $M$-definability of $g(f_1,...,f_n)$ is obvious by substituting formulas. Continuity is also clear, as compositions of continuous functions are continuous.
\end{proof}

\begin{thm}{\label{reflection}}
 $(K_{\mathbb{N}},+,\cdot,exp,<)\equiv_{el}(\mathbb{R},+,\cdot,exp,<)$.
\end{thm}
\begin{proof}
By Wilkie's theorem (see \cite{Wi}), the theory $T_{exp}$ of $\mathbb{R}_{exp}$ is model complete and hence axiomatized by its $A_2$ (i.e. universal existential or $\forall\exists$) -consequences by Proposition $9.3$ from \cite{Sa}. 
It hence suffices to show that every $A_2$-formula that holds in $(\mathbb{R},+,\cdot,exp,<)$ also holds in $(K_{\mathbb{N}},+,\cdot,exp,<)$.\\
So let $\phi$ be an $A_2$-formula in the language of exponential rings that holds in $\mathbb{R}$, say $\phi\equiv\forall{x_1,...,x_n}\exists{y_1,...,y_{m}}(\psi(x_1,...,x_n,y_1,...,y_m)$, where $\psi(x_1,...,x_n,y_1,...,y_m)$ is a Boolean
combination of statements of the form $t(x_1,...,x_n,y_1,..,y_m)=0$ and $t(x_1,...,x_n,y_1,...,y_m)>0$ with $t$ a term in the language of exponential rings.\\
We write $\psi$ in disjunctive normal form, i.e. in the form 
\begin{center} 
$\bigvee_{i=1}^{N}(\bigwedge_{j=1}^{l_{i}}t_{ij}(x_1,...,x_n,y_1,...,y_m)=0 \wedge$ \\ 
$\bigwedge_{j=1}^{k_{i}}t_{ij}^{\prime}(x_1,...,x_n,y_1,...,y_m)>0)$. (*) \end{center}
Note that we can eliminate negation
by rewriting e.g. $(\neg{t=0}\wedge\psi)$ as $(t>0\wedge\psi)\vee(-t>0\wedge\psi)$ or $(\neg{t>0}\wedge\psi)$ as $(t=0\wedge\psi)\vee(-t>0\wedge\psi)$, so we will assume without loss of generality that only positive atomic formulas occur and that
$\psi$ is already written in this form.\\
Now fix $x_{\phi_{1}}^{\vec{v}_{1}},...,x_{\phi_{n}}^{\vec{v}_{n}}\in K_{\mathbb{N}}$. Since $\phi$ holds in $\mathbb{R}$, there exist $r_{1},...,r_{m}\in\mathbb{R}$ such that 
$\mathbb{R}\models\psi(x_{\phi_{1}}^{\vec{v}_{1}},...,x_{\phi_{n}}^{\vec{v}_{n}},r_1,...,r_m)$. Let us assume without loss of generality that it is the first disjunct 
\begin{center}
$\bigwedge_{j=1}^{l_1}t_{1j}(x_{\phi_{1}}^{\vec{v}_{1}},...,x_{\phi_{n}}^{\vec{v}_{n}},r_1,...,r_m)=0\wedge$ \\ $\bigwedge_{j=1}^{k_1}t_{1j}^{\prime}(x_{\phi_{1}}^{\vec{v}_{1}},...,x_{\phi_{n}}^{\vec{v}_{n}},r_1,...,r_m)>0$ 
\end{center}
that is satisfied. By Lemma \ref{cont}, every term in the language of exponential rings gives rise to a continuous function on $K_{\mathbb{N}}$. Hence, the $t_{1j}^{\prime}$ are continuous. 
Therefore, there are rational numbers $q_{1},...,q_{m},p_{1},...,p_{m}$ 
such that $q_{i}<r_{i}<p_{i}$ for all $1\leq i\leq m$, $1\leq j\leq N$ and such that $t_{1j}^{\prime}(x_{\phi_{1}}^{\vec{v}_{1}},...,x_{\phi_{n}}^{\vec{v}_{n}},z_1,...,z_m)>0$ for all 
$(z_{1},...,z_{m})\in\times_{i=1}^{m}[p_{i},q_{i}]$, $1\leq j\leq k_1$.

This holds in particular for all elements of $\mathbb{Q}$. Hence 
\begin{center} 
$\mathbb{N}\models \forall{\vec{v}_{1},...,\vec{v}_{m}}\forall{a_1,...,a_m}\forall{b_1,...,b_m\neq0}\exists\varepsilon>0((\bigwedge_{\iota=1}^{m}(p_{i}<\frac{a_{\iota}}{b_{\iota}}<q_{i})\implies(\bigwedge_{j=1}^{k_{1}}t_{1j}^{\prime}
(x_{\phi_{1}}^{\vec{v}_{1}},...,x_{\phi_{n}}^{\vec{v}_{n}},\frac{a_{1}}{b_{1}},...,\frac{a_m}{b_{m}})>\varepsilon))$
\end{center}
holds for all $n$-tuples of $\mathfrak{L}_{PA}$-formulas.\\
Now we define zeros for the $t_{1j}$ in $\times_{i=1}^{m}[p_{i},q_{i}]$, depending
on $\phi_1,...,\phi_{n}$, but not on the parameters $\vec{v}_{1},...,\vec{v}_{n}$: To do this, we define a sequence $(s_{i})_{i\in\mathbb{N}}$ of $m$-tuples of rational intervalls as follows: $s_{0}:=([p_{j},q_{j}])_{j=1}^{m}$, and, for all $i\geq 0$,
if $s_{i}=([p_{j}^{i},q_{j}^{i}]_{j=1}^{m})$, we let $s_{i+1}$ be the first (in some natural, e.g. lexicographic ordering) of the $2^{m}$ tuples $\{[r_1,s_1],...,[r_m,s_m]\}$ with 
$[r_{j},s_{j}]\in\{[p_{j}^{i},p_{i}^{j}+\frac{q_{i}^{j}}{2}],[p_{i}^{j}+\frac{q_{i}^{j}}{2},q_{i}^{j}]\}$ for all $1\leq j\leq m$ which contains, for every $\bar{n}\in \mathbb{N}$, a tuple of rationals $q_1,...,q_m$ such that 
$|t_{1j}(x_{\phi_{1}}^{\vec{v}_{1}},...,x_{\phi_{n}}^{\vec{v}_{n}},q_1,...,q_m)|<\frac{1}{m}$ for all $1\leq j\leq l_1$.\\
It is easy to see that this sequence of $m$-tuples is definable in $\mathbb{N}$ and converges to a simultanous solution to the $k_1$ equations in question. Hence $\phi$ holds in $K_{\mathbb{N}}$.\\
This implies that the $A_2$-theory of $(\mathbb{R},+,\cdot,exp,<)$ holds in\\ $(K_{\mathbb{N}},+,\cdot,\exp,<)$. By the model completeness of the former, it follows that $\mathbb{R}$ and $K_{\mathbb{N}}$ are elementary
equivalent.
\end{proof}

\begin{thm}{\label{KMisnice}}
 For any $M\models Th(\mathbb{N})$, we have $(K_{M},+,\cdot,exp,<)\equiv_{el}(K_{\mathbb{N}},+,\cdot,exp,<)$.
\end{thm}
\begin{proof}
By Theorem \ref{reflection}, the theory of $(K_{\mathbb{N}},+,\cdot,exp,<)$ is just $T_{exp}$, the theory of real exponentiation. It hence suffices to show that all $A_2$-formulas that hold in $K_{\mathbb{N}}$
also hold in $K_{M}$. Hence, let $\phi$ be an $A_2$-statement as in the proof of Theorem \ref{reflection} and suppose that $K_{\mathbb{N}}\models\phi$. 
This means that, for all $\phi_1,...,\phi_n \in\mathfrak{L}_{PA}$ and all $\vec{v}_{1},...,\vec{v}_{n}\in K_{\mathbb{N}}$, there are $\phi^{\prime}_{1},...,\phi^{\prime}_{m}\in\mathfrak{L}_{PA}$ and 
$\vec{w}_1,...,\vec{w}_m \in K_{\mathbb{N}}$ such that $K_{\mathbb{N}}\models\psi(x_{\phi_{1}}^{\vec{v}_{1}},...,x_{\phi_{n}}^{\vec{v}_{n}},x_{\phi_{1}^{\prime}}^{\vec{w}_{1}},...,x_{\phi_{m}^{\prime}}^{\vec{w}_{m}})$. Note that statements of
the form (*) above and hence of the form 
\begin{center} $\forall{\vec{v}_{1},...,\vec{v}_{n}}\exists{\vec{w}_{1},...,\vec{w}_{m}}\psi(x_{\phi_{1}}^{\vec{v}_{1}},...,x_{\phi_{n}}^{\vec{v}_{n}},x_{\phi_{1}^{\prime}}^{\vec{w}_{1}},...,x_{\phi_{m}^{\prime}}^{\vec{w}_{m}})$ \end{center}
can be expressed as $\mathfrak{L}_{PA}$-formulas:
Basically, the proof of Theorem \ref{reflection} shows that, for every $A_2$-formula $\phi$ as above true in $K_\mathbb{N}$ and every $n$-tuple of $\mathfrak{L}_{PA}$-formulas $\phi_{1},...,\phi_{n}$, there are $\mathfrak{L}_{PA}$-formulas
$\phi_{1}^{\prime},...,\phi_{m}^{\prime}$ such that 
\begin{center} $\forall{\vec{v}_{1},...,\vec{v}_{n}}\exists{\vec{w}_{1},...,\vec{w}_{m}}\psi(x_{\phi_{1}}^{\vec{v}_{1}},...,x_{\phi_{n}}^{\vec{v}_{n}},x_{\phi_{1}^{\prime}}^{\vec{w}_{1}},...,x_{\phi_{m}^{\prime}}^{\vec{w}_{m}})$ \end{center}
holds in $\mathbb{N}$. Consequently, the same holds in $M$. However, the formulas depend on the rational parameters, whose existence has to be carried over to $M$ as well. We achieve this as follows: Let $\psi_{1},...,\psi_{n}$ be $\mathfrak{L}_{PA}$-formulas.
As $\phi$ holds in $K_{\mathbb{N}}$, we have 
\begin{center}
$\mathbb{N}\models\forall{\vec{v}_{1},...\vec{v}_{n}}\exists{a_1,...,a_n,a_{1}^{\prime},...,a_{n}^{\prime}}\exists{b_1,...,b_n,b_{1}^{\prime},...,b_{n}^{\prime} \neq0}\exists{C>0}$ \\
$\bigvee_{i=1}^{N}
(((\forall{c_1,...,c_n}\forall{d_1,...,d_n \neq0}\bigwedge_{j=1}^{n}\frac{a_{j}}{b_{j}}<\frac{c_{j}}{d_{j}}<\frac{a_{j}^{\prime}}{b_{j}^{\prime}}\implies
(\bigwedge_{j=1}^{k_{i}}t_{ij}^{\prime}(x_{\psi_{1}}^{\vec{v}_{1}},...,x_{\psi_{n}}^{\vec{v}_{n}},\frac{c_1}{d_1},...,\frac{c_n}{d_n})>\frac{1}{C}))\wedge$ \\ 
$(\forall{C^{\prime}>0}\exists{c_1,...,c_n}\exists{d_1,...,d_n \neq0}((\bigwedge_{j=1}^{n}\frac{a_{j}}{b_{j}}<\frac{c_{j}}{d_{j}}<\frac{a_{j}^{\prime}}{b_{j}^{\prime}})\wedge$ \\
$(\bigwedge_{j=1}^{l_{i}}|t_{ij}(x_{\psi_{1}}^{\vec{v}_{1}},...,x_{\psi_{n}}^{\vec{v}_{n}},\frac{c_1}{d_1},...,\frac{c_n}{d_n})|<\frac{1}{C^{\prime}})))))$
\end{center}
(i.e. for all choices of the parameters, there are a positive $\varepsilon$ and a rational box $B$ such that, for at least one of the disjoints in $\psi$ (the quantifier-free part of $\phi$, see the proof of Theorem \ref{reflection}),
all $t^{\prime}$ are bigger than $\varepsilon$ in $B$ while the absolute values of the $t$ at rational numbers in $B$ have no positive lower bound).\\
The same statement hence holds in $M$. That $\phi$ holds in $K_{M}$ now follows from the continuity of the $t$ and the $t^{\prime}$ in $K_{M}$.

\end{proof}

We note the following useful consequence of the proof of Theorem \ref{KMisnice}:

\begin{lemma}{\label{uniformisation}}(The $A_2$-uniformisation lemma)\\
Let $K_{M}\models\forall{x_{1},x_{2},...,x_{m}}\exists{y_{1},...,y_{n}}\psi(x_{1},...,x_{m},y_{1},...,y_{n},\vec{v})$, where $\psi$ is quantifier-free and $\vec{v}\subseteq K_{M}$ is finite. 
Then, for every $m$-tuple $(\phi_1,...,\phi_m)$ of formulas in the language $\mathfrak{L}_{exp}$ of ordered exponential rings, there exists an $n$-tuple $(\psi_1,...,\psi_n)$ of $\mathfrak{L}_{exp}$-formulas
such that $K_{M}\models\forall{\vec{v}_{1},...,\vec{v}_{m}}\exists{\vec{w}_{1},...,\vec{w}_{n}}\psi(x_{\phi_{1}}^{\vec{v}_{1}},...,x_{\phi_{m}}^{\vec{v}_{m}},x_{\psi_{1}}^{\vec{w}_{1}},...,x_{\psi_{n}}^{\vec{w}_{n}},\vec{v})$.
\end{lemma}
\begin{proof}
The proof of Theorem \ref{KMisnice} shows how to obtain such formulas.
\end{proof}

\begin{corollary}
For every $M\models Th(\mathbb{N})$, we have 
\begin{center} $(K_{M},+,\cdot,exp,<)\equiv_{el}(\mathbb{R},+,\cdot,exp,<)$.\end{center}
 Consequently, every model of true arithmetic is an $IP$ of a real closed exponential field modelling $T_{exp}$.
\end{corollary}
\begin{proof}
 Immediate from Theorem \ref{reflection}, Theorem \ref{KMisnice} and the fact that $M$ is an $IP$ of $K_{M}$ that we proved above.
\end{proof}

\section{A counterexample for bounded arithmetic}

The results of the preceeding sections about true arithmetic and Peano Arithmetic stand in sharp contrast with the situation for the weaker fragment of bounded arithmetic ($I\Delta_{0}$). In this case, already quite weak 
notions of exponential may fail to occur.

\begin{thm}
 There is a model $M\models I\Delta_{0}$ such that, for no $RCF$ $K$ which has $M$ as an $IP$, there exists $g:K\rightarrow K$ such that \begin{center}$(K,+,\cdot,g,<)\equiv_{el}(\mathbb{R},+,\cdot,exp,<)$.\end{center}
\end{thm}
\begin{proof}
 Let $M$ be a bounded nonstandard model of $I\Delta_{0}$, i.e. there exists a nonstandard element $a\in M$ such that $\{a^{i}|i\in\omega\}$ is cofinal in $M$. As $(\mathbb{R},+,\cdot,exp,<)\models\forall{x>1}(exp(x)>x)$,
we have that $g(a)>a$. Also, $exp$ is monotonic and hence $g$ is monotonic. It follows that $g(a^{2})>a^{i}$ for all $i\in\omega$, as, in fact, we have $g(a^{2})=g(aa)>g(ax)$ for all $x<a$, and since $a$ is nonstandard, this holds in particular
for all finite $x$. But as $\{a^{i}|i\in\omega\}$ is cofinal in $K$, such an element does not exist in $K$, hence exponentiation is not total in $K$, a contradiction.
\end{proof}

In fact, we can strengthen this further. The following definition comes from \cite{Ku}.

\begin{defini}
Let $K$ be an $RCF$. A $GA$-exponential $f$ on $K$ is an isomorphism between $(K,+,<)$ and $(K^{>0},\cdot,<)$ such that, for all $a\in K$ and $n\in\mathbb{N}$, we have that $a\geq n^{2}$ implies $f(a)>a^{n}$.
\end{defini}

\begin{thm}
 There is a model $M\models I\Delta_{0}$ such that, for no $RCF$ $K$ which has $M$ as an $IP$, there exists $g:K\rightarrow K$ such that $g$ is a $GAT$-exponential on $K$.
\end{thm}
\begin{proof}
Let $M$ again be a bounded nonstandard model of $I\Delta_{0}$ as above and assume for a contradiction that $f$ is a $GA$-exponential on $M$. Let $a\in M$ be nonstandard such that $\{a^{i}|i\in\mathbb{N}\}$ is cofinal
in $M$ and hence in $K$. As $a$ is nonstandard, we have $a>n^{2}$ for all $n\in\mathbb{N}$. Hence, since $f$ is a $GA$-exponential, we have $f(a)>a^{n}$ for every $n\in\mathbb{N}$. But this implies that $f(a)$ is strictly
greater than every element of $K$, a contradiction.
\end{proof}

\section{Further work}

The arguments from section $3$ can be extended to any functions that are definable over $M$ and preserve the model-completeness.

\end{document}